\documentclass[11pt]{article}
\usepackage{graphicx}
\usepackage{epstopdf}
\usepackage{amsmath,amsthm,amssymb,amscd}
\usepackage{ascmac}
\usepackage[arrow,matrix]{xy}
\usepackage{enumerate}
\usepackage{color}

\newtheorem{thm}{Theorem}[section]

\newtheorem{prop}[thm]{Proposition}

\newtheorem{lem}[thm]{Lemma}

\newtheorem{rem}[thm]{Remark}

\newtheorem{que}[thm]{Question}

\title{An infinite family of prime knots with a certain property for the clasp number}

\author{Teruhisa KADOKAMI and Kengo KAWAMURA}
\date{}                                           % Activate to display a given date or no date

\begin{document}

\maketitle

%%%%%%%%%%%%%%%%%%%%%%%%%%%%%%%%%%%%%%%%%%%%%%%%%%%%%%%%%%%%%%%%%%%%%%%%%%%%%%%%%%%%%%%%%%%%%%%%%%

\begin{abstract}

The clasp number $c(K)$ of a knot $K$ is the minimum number of clasp singularities among all clasp disks bounded by $K$.
It is known that the genus $g(K)$ and the unknotting number $u(K)$ are lower bounds of the clasp number, that is, $\max\{g(K),u(K)\} \leq c(K)$.
Then it is natural to ask whether there exists a knot $K$ such that $\max\{g(K),u(K)\}<c(K)$.
In this paper, we prove that there exists an infinite family of prime knots such that the question above is affirmative.

\end{abstract}

%%%%%%%%%%%%%%%%%%%%%%%%%%%%%%%%%%%%%%%%%%%%%%%%%%%%%%%%%%%%%%%%%%%%%%%%%%%%%%%%%%%%%%%%%%%%%%%%%%

\section{Introduction}

It is known that every knot in $S^{3}$ bounds a singular disk in $S^{3}$ whose singular set consists of only clasp singularities as illustrated in Figure~\ref{fig:clasp}.
\begin{figure}[h]
 \centering
 \includegraphics{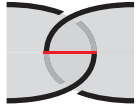}
 \caption{A clasp singularity}
 \label{fig:clasp}
\end{figure}
We call such a singular disk a {\it clasp disk} of the knot.
Let $K$ be a knot in $S^{3}$, and $D$ a clasp disk of $K$.
Let $c(D)$ denote the number of clasp singularities in $D$.
Then the {\it clasp number} of $K$ is $c(K)=\min \{\,c(D)\,|\,\mbox{$D$ is a clasp disk of $K$}\,\}$.
We refer the reader to \cite{Co,Hi,Ka,Mh,Mo1,Mo2,Mo3,Sh} for related topics of the clasp number.
In this paper, we suppose that every link is in $S^{3}$ and oriented, and the notation of prime knots follows Rolfsen's book \cite{Ro}.

For a knot $K$, let $g(K)$ and $u(K)$ be the genus and the unknotting number of $K$ respectively.
In \cite{Sh}, T.~Shibuya proved that the genus and the unknotting number are lower bounds of the clasp number.
In other words, for a knot $K$ we have $\max\{g(K),u(K)\} \leq c(K)$.
Most of the prime knots with up to $10$ crossings satisfy the equality above (cf.\ Appendix).
Then it is natural to ask the following question.
\begin{que}\label{que1}
Does there exist a prime knot $K$ such that $\max\{g(K),u(K)\}<c(K)$?
\end{que}
For an integer $n$, let $K_{n}$ be the knot as illustrated in Figure~\ref{fig:ori-k_n}.
\begin{figure}[h]
 \centering
 \includegraphics{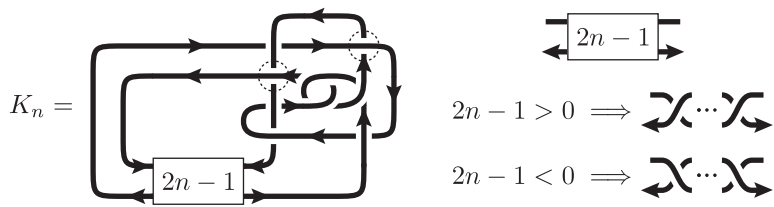}
 \caption{The knot $K_{n}$}
 \label{fig:ori-k_n}
\end{figure}
Here, an integer in the rectangle denotes the number of half twists.
We note that $K_{0}=\overline{3_{1}}\#4_{1}$ (Figure~\ref{fig:k_0}) and $K_{1}=10_{97}$ (Figure~\ref{fig:k_1}).
\begin{figure}[h]
\begin{minipage}{100mm}
 \centering
 \includegraphics{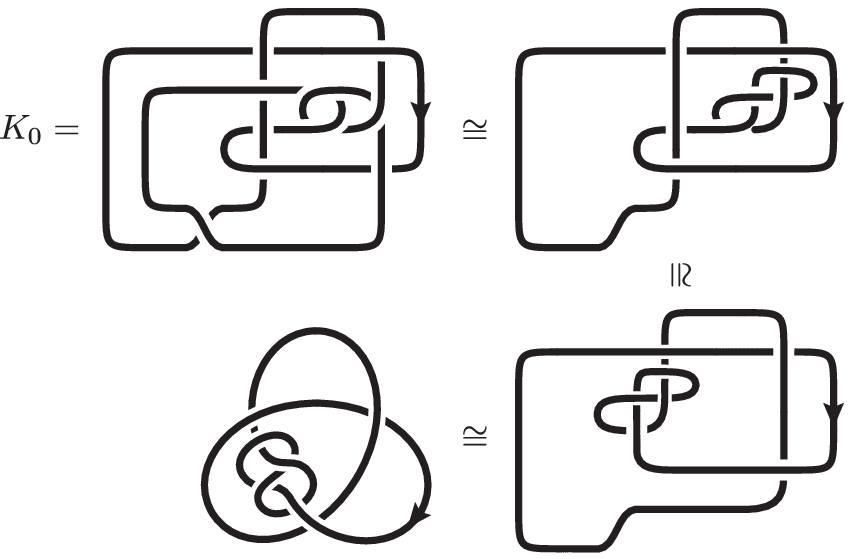}
 \caption{$K_{0}=\overline{3_{1}}\#4_{1}$}
 \label{fig:k_0}
\end{minipage}
\begin{minipage}{40mm}
 \centering
 \includegraphics{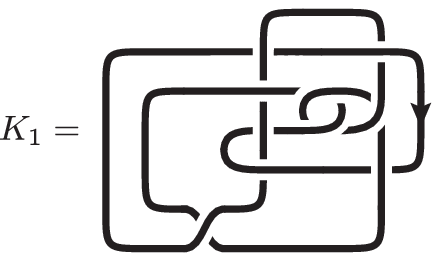}
 \caption{$K_{1}=10_{97}$}
 \label{fig:k_1}
\end{minipage}
\end{figure}
Note that all the knots $K_{n}$ are distinct one another (Proposition~\ref{prop:Conway}).
The main result of this paper is the following theorem.
It is the affirmative answer of Question~\ref{que1}.
\begin{thm}\label{thm1}
If $n$ is odd, then the knots $K_{n}$ are prime and we have $\max\{g(K_{n}),u(K_{n})\}<c(K_{n})$.
\end{thm}

In Appendix, we show a table for the clasp numbers of prime knots with up to $10$ crossings.
From the table, we raise the following question concerning an upper bound of the clasp number.
\begin{que}\label{que2}
Let $cr(K)$ be the crossing number of a knot $K$.
For any non-trivial knot $K$, does the following inequality hold?
\[
c(K)\leq \left[\frac{cr(K)-1}{2}\right],
\]
where $[r]$ denotes the integer part of a rational number $r$.
\end{que}
Question~\ref{que2} is affirmative for prime knots with up to $10$ crossings (cf.\ Appendix) and the knots $K_{n}$ (Proposition~\ref{prop:que2}).
In Section~\ref{sec:suff}, we calculate the Conway polynomial of a knot $K$ with $c(K)\leq2$, giving an alternative proof of Morimoto's result \cite{Mo3} (Lemma~\ref{lem:Conway}).
(Since it seems to be hard to obtain a reference \cite{Mo3}, it is also our purpose to expose it to the reader.)
Moreover we give a sufficient condition for a knot $K$ to satisfy $c(K) \geq 3$.
In Section~\ref{sec:proof}, we investigate the knots $K_{n}$ and prove Theorem~\ref{thm1},
and we also prove that Question~\ref{que2} is affirmative for the knots $K_{n}$.

%%%%%%%%%%%%%%%%%%%%%%%%%%%%%%%%%%%%%%%%%%%%%%%%%%%%%%%%%%%%%%%%%%%%%%%%%%%%%%%%%%%%%%%%%%%%%%%%%%

\section{A sufficient condition for a knot $K$ to satisfy $c(K) \geq 3$}\label{sec:suff}

In \cite{Mo3}, K.~Morimoto calculated the Alexander module of a knot from a clasp disk.
In this section, we calculate the Conway polynomial of a knot $K$ with $c(K)\leq2$, giving an alternative proof of Morimoto's result \cite{Mo3}, and then we give a sufficient condition for a knot $K$ to satisfy $c(K) \geq 3$.

First, we prove the following lemma.
\begin{lem}\label{lem:Conway}{\rm (cf.\ \cite{Mo3})}
Let $K$ be a knot with $c(K) \leq 2$.
The Conway polynomial $\nabla_{K}(z)$ of $K$ is expressed as follows:
\[
\begin{array}{r}
\nabla_{K}(z)=\bigl(b_{1}b_{2}+\varepsilon b_{3}(b_{3}+\delta)\bigr)z^{4}
+(b_{1}+b_{2}-\varepsilon\delta)z^{2}+1~\medskip\\
(b_{1}, b_{2}, b_{3}\in \mathbb{Z},\ \varepsilon\in\{\pm 1\},\ \delta\in\{0,1\}).
\end{array}
\]
\end{lem}
\begin{proof}
Let $K$ be a knot with $c(K)\leq2$.
Then there exists a clasp disk $D$ of $K$ such that $c(D)=2$.
We may assume that $D$ is a surface which is the union of a disk $B_{0}$ and two clasping bands $B_{1}$ and $B_{2}$ (see Figure~\ref{fig:example}).
Here, a clasping band means a pair of embedded $2$-disks in $S^{3}$ with a clasp singularity as illustrated in Figure~\ref{fig:clasping_band}.
The sign of a clasping band is defined as the linking number of the Hopf link bounding the clasping band.
\begin{figure}[h]
 \centering
 \includegraphics{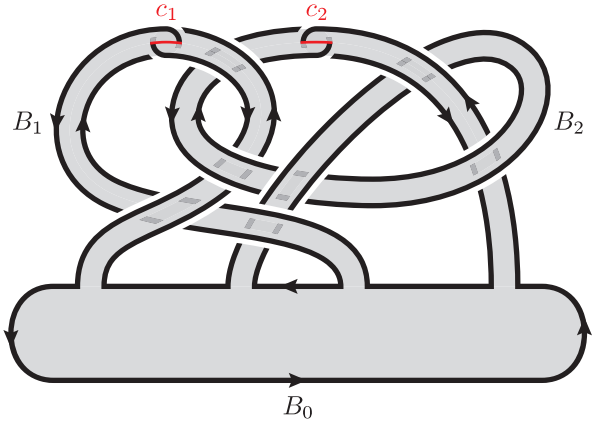}
 \caption{Example of a clasp disk}
 \label{fig:example}
\end{figure}
\begin{figure}[h]
 \centering
 \includegraphics{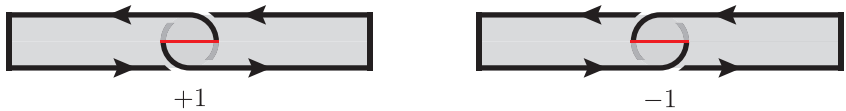}
 \caption{Positive and negative clasping bands}
 \label{fig:clasping_band}
\end{figure}
For each $i$ $(i=1,2)$, let $c_{i}$ be the clasp singularity in $D$ corresponding to the clasping band $B_{i}$.
We define the sign $\varepsilon_{i}\in\{\pm1\}$ of the clasp singularity $c_{i}$ as the sign of $B_{i}$.
For each $\varepsilon\in\{\pm1\}$, an $(\varepsilon)$-Hopf band is an annulus in $S^{3}$ whose boundary is a Hopf link with linking number $\varepsilon$.
Let $F$ be the oriented surface obtained from $D$ by the following operation:
\begin{enumerate}
\item[($*$)] For each $i$ $(i=1,2)$, replacing the neighborhood of the clasp singularity $c_{i}$ with a plumbing of an $(\varepsilon_{i})$-Hopf band, denoted by $H_{i}$, as illustrated in Figure~\ref{fig:replacing}.
\end{enumerate}
Note that the surface $F$ is a genus two Seifert surface of $K$.
For example, the Seifert surface in Figure~\ref{fig:induced_surface} is obtained from the clasp disk in Figure~\ref{fig:example} by the operation $(*)$.
We take a homological basis $\{[\alpha_{1}], [\alpha_{2}], [\beta_{1}], [\beta_{2}]\}$ of $H_{1}(F;\mathbb{Z})$ as follows (see Figure~\ref{fig:induced_surface}):
\begin{enumerate}
\item[(1)] For each $i$ $(i=1,2)$, $\alpha_{i}$ is a loop corresponding to the core of the clasping band $B_{i}$ and $\beta_{i}$ is the core loop of the Hopf band $H_{i}$.
\item[(2)] Orientations of $\alpha_{i}$ and $\beta_{i}$ are chosen so that the intersection number $[\beta_{i}]\cdot[\alpha_{i}]$ is $1$.
\end{enumerate}
\begin{figure}[h]
 \centering
 \includegraphics{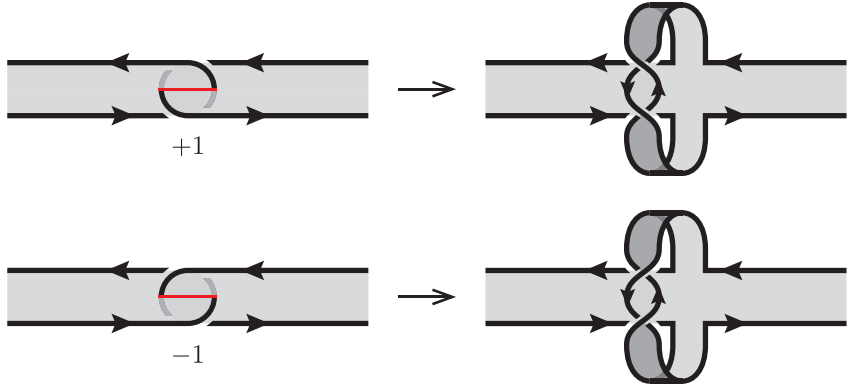}
 \caption{A construction of a Seifert surface $F$ from a clasp disk $D$}
 \label{fig:replacing}
\end{figure}
\begin{figure}[h]
 \centering
 \includegraphics{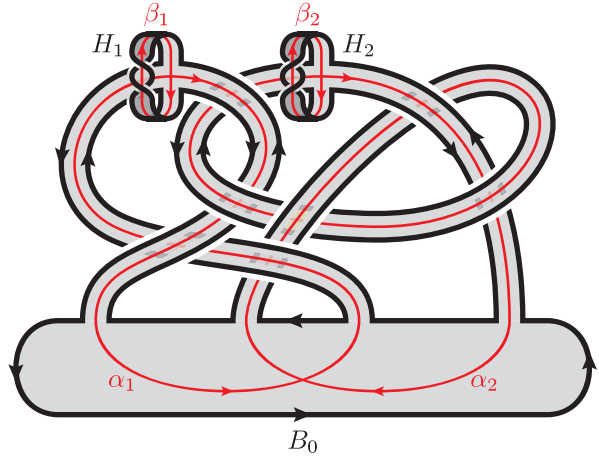}
 \caption{The Seifert surface obtained from the clasp disk in Figure~\ref{fig:example} and a homological basis $\{[\alpha_{1}], [\alpha_{2}], [\beta_{1}], [\beta_{2}]\}$}
 \label{fig:induced_surface}
\end{figure}

Let $a_{ij}=\mathrm{lk}(\alpha_{i},\alpha_{j}^{+})$ be the linking number of $\alpha_{i}$ and $\alpha_{j}^{+}$, where $\alpha_{j}^{+}$ is a loop obtained by pushing $\alpha_{j}$ to the positive normal direction of $F$.
The Seifert matrix $V$ of $K$ obtained from the homological basis $\{[\alpha_{1}], [\alpha_{2}], [\beta_{1}], [\beta_{2}]\}$ is calculated as follows:
\[
V=
\left(
\begin{array}{cccc}
 a_{11} & a_{12} &                0 &                0\\[1mm]
 a_{21} & a_{22} &                0 &                0\\[1mm]
     -1 &      0 & -\varepsilon_{1} &                0\\[1mm]
      0 &     -1 &                0 & -\varepsilon_{2}
\end{array}
\right)\text{.}
\]
Therefore the Alexander polynomial $\Delta_{K}(t)$ of $K$ is calculated as follows:
\begin{align*}
\Delta_{K}(t)
&
\doteq\det(tV-V^{\rm T})\\[1mm]
&
=\det
\left(
\begin{array}{cccc}
   a_{11}(t-1) & a_{12}t-a_{21} &                     1 &                     0\\[1mm]
a_{21}t-a_{12} &    a_{22}(t-1) &                     0 &                     1\\[1mm]
            -t &              0 & -\varepsilon_{1}(t-1) &                     0\\[1mm]
             0 &             -t &                     0 & -\varepsilon_{2}(t-1)
\end{array}
\right)\\[1mm]
%&
%=\det
%\left(
%\begin{array}{cccc}
%                         a_{11}(t-1) &                       a_{12}t-a_{21} & 1 & 0\\[1mm]
%                      a_{21}t-a_{12} &                          a_{22}(t-1) & 0 & 1\\[1mm]
%    \varepsilon_{1}a_{11}(t-1)^{2}-t & \varepsilon_{1}(t-1)(a_{12}t-a_{21}) & 0 & 0\\[1mm]
%\varepsilon_{2}(t-1)(a_{21}t-a_{12}) &     \varepsilon_{2}a_{22}(t-1)^{2}-t & 0 & 0
%\end{array}
%\right)\\[1mm]
&
=\det
\left(
\begin{array}{cc}
 \varepsilon_{1}a_{11}(t-1)^{2}-t & \varepsilon_{1}(t-1)(a_{12}t-a_{21}) \\[1mm]
 \varepsilon_{2}(t-1)(a_{21}t-a_{12}) & \varepsilon_{2}a_{22}(t-1)^{2}-t 
\end{array}
\right)\\[1mm]
&
=\bigl(\varepsilon_{1}a_{11}(t-1)^{2}-t\bigr)\bigl(\varepsilon_{2}a_{22}(t-1)^{2}-t\bigr)-\varepsilon_{1}\varepsilon_{2}(t-1)^{2}(a_{12}t-a_{21})(a_{21}t-a_{12})\\[1mm]
&
=\varepsilon_{1}\varepsilon_{2}\bigl(a_{11}a_{22}-a_{12}a_{21}\bigr)(t-1)^{4}+\bigl(-\varepsilon_{1}a_{11}-\varepsilon_{2}a_{22}+\varepsilon_{1}\varepsilon_{2}(a_{21}-a_{12})^{2}\bigr)t(t-1)^{2}+t^{2}.
\end{align*}
(The symbol $\doteq$ means $``$is equal to, up to multiplication by a unit of $\mathbb{Z}[t,t^{-1}]$$"$, and $V^{\rm T}$ means the transpose of $V$.)
Note that if $\alpha_{1}\cap\alpha_{2}\neq\emptyset$, then we may assume that orientations of $\alpha_{1}$ and $\alpha_{2}$ are chosen so that the intersection number $[\alpha_{1}]\cdot[\alpha_{2}]$ is $1$ (see Figure~\ref{fig:induced_surface}).
We put $b_{1},b_{2},b_{3},\varepsilon$ and $\delta$ as follows:
\begin{align*}
b_{1}&:=-\varepsilon_{1}a_{11},~b_{2}:=-\varepsilon_{2}a_{22},~b_{3}:=a_{12},\\[1mm]
\varepsilon&:=-\varepsilon_{1}\varepsilon_{2}~\mbox{and}~
\delta:=a_{21}-a_{12}=
\left\{
\begin{array}{ll}
0 & (\alpha_{1}\cap\alpha_{2}=\emptyset),\\[1mm]
1 & (\alpha_{1}\cap\alpha_{2}\neq\emptyset).
\end{array}
\right.
\end{align*}
By these substitutions and $\delta^{2}=\delta$, we have the following:
\[
\begin{array}{r}
\Delta_{K}(t)
\doteq\bigl(b_{1}b_{2}+\varepsilon b_{3}(b_{3}+\delta)\bigr)(t-1)^{4}
+(b_{1}+b_{2}-\varepsilon\delta)t(t-1)^{2}+t^{2}~\medskip\\
(b_{1}, b_{2}, b_{3}\in \mathbb{Z},\ \varepsilon\in\{\pm 1\},\ \delta\in\{0,1\}).
\end{array}
\]

For the Conway polynomial $\nabla_{K}(z)$ of $K$, since $\nabla_{K}(t^{-\frac{1}{2}}-t^{\frac{1}{2}})=\det(t^{\frac{1}{2}}V-t^{-\frac{1}{2}}V^{\rm T})=t^{-2}\det(tV-V^{\rm T})$ and $z^{2}=\displaystyle\frac{(t-1)^{2}}{t}$, we obtain the following:
\[
\nabla_{K}(z)=\bigl(b_{1}b_{2}+\varepsilon b_{3}(b_{3}+\delta)\bigr)z^{4}
+(b_{1}+b_{2}-\varepsilon\delta)z^{2}+1.
\]
\end{proof}

We mention that if a knot $K$ has $c(K)=1$, then we obtain $\nabla_{K}(z)=b_{1}z^{2}+1$ by substitutions $b_{2}=0,b_{3}=0~\mbox{and}~\delta=0$.

%We mention the Conway polynomial of a clasp number one knot.
%Let $K$ be a clasp number one knot and $D$ a clasp disk of $K$.
%Suppose that $c(K)=c(D)=1$.
%We may assume that $D$ is a surface which the union of a disk $B_{0}$ and a clasping band $B_{1}$.
%By a fundamental move (cf.\ \cite{Co,Ka}) as illustrated in Figure~\ref{}, we create a clasping band $B_{2}$ on $D$ and obtain a clasp disk $D'=D \cup B_{2}=B_{0} \cup B_{1} \cup B_{2}$ of $K$.
%Then $c(D')=2$ and we have $\nabla_{K}(z)=b_{1}z^{2}+1$ by substitutions $b_{2}=0$, $b_{3}=0$ and $\delta=0$.

By Lemma~\ref{lem:Conway}, we obtain the following.
\begin{prop}\label{prop:suff}
Suppose that a knot $K$ has $\nabla_{K}(z)=m_{4}z^{4}+m_{2}z^{2}+1$.
If $m_{4}\equiv 3 \pmod{8}$ and $m_{2}\equiv 2 \pmod{4}$, then $c(K) \geq 3$.
\end{prop}
\begin{proof}
Suppose that a knot $K$ has $c(K) \leq 2$ and $\nabla_{K}(z)=m_{4}z^{4}+m_{2}z^{2}+1$ with $m_{4}\equiv 3 \pmod{8}$ and $m_{2}\equiv 2 \pmod{4}$.
By Lemma~\ref{lem:Conway}, the following holds:
\[
m_{4}=b_{1}b_{2}+\varepsilon b_{3}(b_{3}+\delta)~~\mbox{and}~~m_{2}=b_{1}+b_{2}-\varepsilon\delta
\]
for some $b_{1},b_{2},b_{3}\in\mathbb{Z}$, $\varepsilon\in\{\pm1\}$ and $\delta\in\{0,1\}$.

{\it The case of $\delta=0$}:\
We obtain $m_{4}=b_{1}b_{2}+\varepsilon b_{3}^{2}$ and $m_{2}=b_{1}+b_{2}$.
Since $b_{1}+b_{2}$ is even, $b_{1}-b_{2}$ is also even.
Therefore,
\[
d:=\left(\frac{b_{1}-b_{2}}{2}\right)^{2}=\frac{1}{4}m_{2}^{2}-m_{4}+\varepsilon b_{3}^{2}\equiv 6+\varepsilon b_{3}^{2} \pmod{8}.
\]
Then $d\equiv 2, 5, 6~\mbox{or}~7 \pmod{8}$, and $d$ cannot be a square integer.
This is a contradiction.

\medskip

{\it The case of $\delta=1$}:\
We obtain $m_{4}=b_{1}b_{2}+\varepsilon b_{3}(b_{3}+1)$ and $m_{2}=b_{1}+b_{2}-\varepsilon$.
\begin{align*}
d':=(b_{1}-b_{2})^{2}
&
=(m_{2}+\varepsilon)^{2}-4\bigl(m_{4}-\varepsilon b_{3}(b_{3}+1)\bigr)\\
&
\equiv 1-4 \pmod{8}\\
&
\equiv 5 \pmod{8}.
\end{align*}
Then $d'$ cannot be a square integer.
This is a contradiction.
\end{proof}
%%

%%%%%%%%%%%%%%%%%%%%%%%%%%%%%%%%%%%%%%%%%%%%%%%%%%%%%%%%%%%%%%%%%%%%%%%%%%%%%%%%%%%%%%%%%%%%%%%%%%

\section{Proof of Theorem~\ref{thm1}}\label{sec:proof}

In this section, first we calculate the Conway polynomial (Proposition~\ref{prop:Conway}) and the Jones polynomial (Proposition~\ref{prop:Jones}) of the knot $K_{n}$ by using the skein relation.
Next, we provide a lemma (Lemma~\ref{lem1}) and then we prove Theorem~\ref{thm1}.

Let $J$ be the oriented link as illustrated in Figure~\ref{fig:link}.
It is obtained from the knot $K_{n}$ by smoothing one of $|2n-1|$ crossings in the rectangle in Figure~\ref{fig:ori-k_n}.
We note that it is equivalent to $L9a41\{1\}$ in LinkInfo table \cite{CL2}.
\begin{figure}[h]
 \centering
 \includegraphics{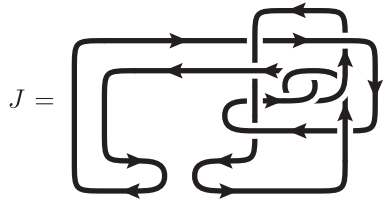}
 \caption{The link $J=L9a41\{1\}$}
 \label{fig:link}
\end{figure}
%

%% %% %% %% %% %% %% %% %% %%
\subsection{The Conway polynomial of the knot $K_{n}$}\label{subsec:Conway}

The Conway polynomial $\nabla_{L}(z)\in\mathbb{Z}[z]$ of an oriented link $L$ is characterized by the following skein relation:
\[
{\rm (i)}~\nabla_{\rm unknot}(z)=1\qquad
{\rm (ii)}~\nabla_{L_{+}}(z)-\nabla_{L_{-}}(z)=z\nabla_{L_{0}}(z){\rm,}
\]
where three oriented links $L_{+},L_{-}$ and $L_{0}$ coincide except in the neighborhood of a point as illustrated in Figure~\ref{fig:skein}.
\begin{figure}[h]
 \centering
 \includegraphics{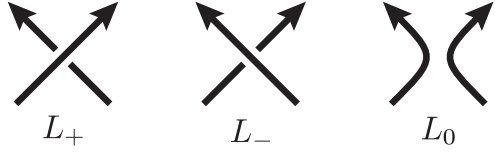}
 \caption{A skein triple $(L_{+},L_{-},L_{0})$}
 \label{fig:skein}
\end{figure}
We call $(L_{+},L_{-},L_{0})$ a skein triple.
We can calculate the Conway polynomial $\nabla_{K_{n}}(z)$ of the knot $K_{n}$ as follows.
\begin{prop}\label{prop:Conway}
The Conway polynomial $\nabla_{K_{n}}(z)$ of the knot $K_{n}$ is as follows:
\[
\nabla_{K_n}(z)=-(4n+1)z^4+2nz^2+1.
\]
\end{prop}
\begin{proof}
{\it The case of $n=0$}:
Since $K_{0}=\overline{3_{1}}\#4_{1}$, 
\[
\nabla_{K_{0}}(z)=\nabla_{\overline{3_{1}}}(z)\nabla_{4_{1}}(z)=(z^{2}+1)(-z^{2}+1)=-z^{4}+1.
\]
%%
%This is affirmative.

\medskip
\noindent
{\it The case of $n>0$}:
The Conway polynomial of $J$ is $\nabla_{J}(z)=-4z^{3}+2z$.
%For a non-zero integer $n$, let $\sigma:=\frac{\displaystyle n}{\displaystyle |n|}\in\{\pm1\}$ denote the sign of $n$.
By considering a skein triple $(K_{n},K_{n-1},J)$, the Conway polynomial of $K_{n}$ can be calculated inductively as follows:
\begin{align*}
\nabla_{K_{n}}(z)&=\nabla_{K_{n-1}}(z)+z\nabla_{J}(z)\\
                 &=(\nabla_{K_{n-2}}(z)+z\nabla_{J}(z))+z\nabla_{J}(z)\\
                 &\cdots\\
                 &=\nabla_{K_{0}}(z)+nz\nabla_{J}(z)\\
                 &=(-z^{4}+1)+nz(-4z^{3}+2z)=-(4n+1)z^{4}+2nz^{2}+1.
\end{align*}
Similarly, we can calculate it in the case of $n<0$.
\end{proof}
By Proposition~\ref{prop:Conway}, we see that all the knots $K_{n}$ are distinct one another.

\subsection{The Jones polynomial of $K_{n}$}\label{subsec:Jones}

The Jones polynomial $V_{L}(t) \in \mathbb{Z}[t^{\frac{1}{2}},t^{-\frac{1}{2}}]$ of an oriented link $L$ is characterized by the following skein relation:
\[
{\rm (i)}~V_{\rm unknot}(t)=1\qquad
{\rm (ii)}~t^{-1}V_{L_{+}}(t)-tV_{L_{-}}(t)=(t^{\frac{1}{2}}-t^{-\frac{1}{2}})V_{L_{0}}(t)
\]
for any skein triple $(L_{+},L_{-},L_{0})$.
We can calculate the Jones polynomial $V_{K_{n}}(t)$ of the knot $K_{n}$ as follows.
\begin{prop}\label{prop:Jones}
The Jones polynomial $V_{K_{n}}(t)$ of the knot $K_{n}$ is as follows:
\[
V_{K_{n}}(t)=\left\{
\begin{array}{ll}
t^{-1}-1+2t-3t^{2}+3t^{3}-2t^{4}+2t^{5}-t^{6}
&
(n=0)\\[2mm]
t^{2n}V_{K_{0}}(t)+\sigma t^{\sigma}(t^{\frac{1}{2}}-t^{-\frac{1}{2}})V_{J}(t)(1+t^{2\sigma}+\dots+t^{2(n-\sigma)})
&
(n\neq0)\text{,}
\end{array}
\right.
\]
where $\sigma:=\frac{\displaystyle n}{\displaystyle |n|}\in\{\pm1\}$ for $n\neq0$.
\end{prop}
\begin{proof}
{\it The case of $n=0$}:
Since $K_{0}=\overline{3_{1}}\#4_{1}$, 
\begin{align*}
V_{K_{0}}(t)=V_{\overline{3_{1}}}(t)V_{4_{1}}(t)&=(t+t^{3}-t^{4})(t^{-2}-t^{-1}+1-t+t^{2})\\
                                     &=t^{-1}-1+2t-3t^{2}+3t^{3}-2t^{4}+2t^{5}-t^{6}.
\end{align*}

\medskip
\noindent
{\it The case of $n>0$}: By considering a skein triple $(K_{n},K_{n-1},J)$, the Jones polynomial of $K_{n}$ can be calculated inductively as follows:
\begin{align*}
V_{K_{n}}(t)&=t^{2}V_{K_{n-1}}(t)+t(t^{\frac{1}{2}}-t^{-\frac{1}{2}})V_{J}(t)\\
            &=t^{2}\bigl(t^{2}V_{K_{n-2}}(t)+t(t^{\frac{1}{2}}-t^{-\frac{1}{2}})V_{J}(t)\bigr)+t(t^{\frac{1}{2}}-t^{-\frac{1}{2}})V_{J}(t)\\
            &\cdots\\
            &=t^{2n}V_{K_{0}}(t)+t(t^{\frac{1}{2}}-t^{-\frac{1}{2}})V_{J}(t)(1+t^{2}+\dots+t^{2(n-1)}).
\end{align*}
Similarly, we can calculate it in the case of $n<0$.
\end{proof}
%%

%%%
%\begin{prop}
%\[
%V_{K_{n}}(t)=\left\{
%\begin{array}{ll}
%t^{2n}V_{K_{0}}(t)-t(t^{-\frac{1}{2}}-t^{\frac{1}{2}})V_{J}(t)(1+t^{2}+\dots+t^{2(n-1)})
%&
%(n \geq 1)\\[2mm]
%%
%V_{K_{0}}(t)
%&
%(n=0)\\[2mm]
%%
%t^{2n}V_{K_{0}}(t)+t^{-1}(t^{-\frac{1}{2}}-t^{\frac{1}{2}})V_{J}(t)(1+t^{-2}+\dots+t^{2(n+1)})
%&
%(n \leq -1)
%\end{array}
%\right.
%\]
%\end{prop}
%%%
%%%
%\begin{proof}
%{\it The case of $n\geq1$}: The Jones polynomial $V_{K_{n}}(t)$ of $K_{n}$ can be calculated inductively by considering a skein triple $(K_{n},K_{n-1},J)$.
%\begin{align*}
%V_{K_{n}}(t)&=t^{2}V_{K_{n-1}}(t)-t(t^{-\frac{1}{2}}-t^{\frac{1}{2}})V_{J}(t)\\
%            &=t^{2}\bigl(t^{2}V_{K_{n-2}}(t)-t(t^{-\frac{1}{2}}-t^{\frac{1}{2}})V_{J}(t)\bigr)-t(t^{-\frac{1}{2}}-t^{\frac{1}{2}})V_{J}(t)\\
%            &\cdots\\
%            &=t^{2n}V_{K_{0}}(t)-t(t^{-\frac{1}{2}}-t^{\frac{1}{2}})V_{J}(t)(1+t^{2}+\dots+t^{2(n-1)})
%\end{align*}
%\noindent
%{\it The case of $n\leq-1$}: The Jones polynomial $V_{K_{n}}(t)$ of $K_{n}$ can be calculated inductively by considering a skein triple $(K_{n+1},K_{n},J)$.
%\begin{align*}
%V_{K_{n}}(t)&=t^{-2}V_{K_{n+1}}(t)+t^{-1}(t^{-\frac{1}{2}}-t^{\frac{1}{2}})V_{J}(t)\\
%            &=t^{-2}\bigl(t^{-2}V_{K_{n+2}}(t)+t^{-1}(t^{-\frac{1}{2}}-t^{\frac{1}{2}})V_{J}(t)\bigr)+t^{-1}(t^{-\frac{1}{2}}-t^{\frac{1}{2}})V_{J}(t)\\
%            &\cdots\\
%            &=t^{2n}V_{K_{0}}(t)+t^{-1}(t^{-\frac{1}{2}}-t^{\frac{1}{2}})V_{J}(t)(1+t^{-2}+\dots+t^{2(n+1)})
%\end{align*}
%\end{proof}
%%%

%% %% %% %% %% %% %% %% %% %%
\subsection{Proof of Theorem~\ref{thm1}}\label{subsec:proof}

First, we prove the following lemma.
\begin{lem}\label{lem1}
For the knot $K_{n}$ $(n\in\mathbb{Z})$, we have the following:
\begin{enumerate}
\item[(1)] $g(K_{n})=2$.
\item[(2)] $u(K_{n})\leq2$.
\item[(3)] $2 \leq c(K_{n}) \leq 4$. In particular, $c(K_{n})\ge 3$ for odd $n$.
\item[(4)] $K_{n}$ is prime for $n\neq0, -4$.
\end{enumerate}
\end{lem}
\begin{proof}
(1)~Since we can obtain a genus two Seifert surface of $K_{n}$ as illustrated in Figure~\ref{fig:Seifert_surface}, we have $g(K_{n})\leq2$.
Since the degree of $\nabla_{K_{n}}(z)$ is four (Proposition~\ref{prop:Conway}), we have $g(K_{n})\geq2$.
Therefore we obtain $g(K_{n})=2$.

\medskip
\noindent
(2)~Since we can unknot the knot $K_{n}$ by crossing changes at dotted circles in Figure~\ref{fig:ori-k_n}, we have $u(K_{n})\leq2$.

\medskip
\noindent
(3)~By (1) we have $c(K_{n})\geq2$, and by Figure~\ref{fig:c=4} we have $c(K_{n})\leq4$.
Suppose that $n$ is odd.
Since $\nabla_{K_{n}}(z)=-(4n+1)z^{4}+2nz^{2}+1$, it implies that $-(4n+1)\equiv3 \pmod{8}$ and $2n\equiv2 \pmod{4}$.
Therefore by Proposition~\ref{prop:suff} we obtain $c(K_{n})\geq3$.

\medskip
\noindent
(4)~Suppose that $K_{n}$ is a composite knot.
By (1), $K_{n}$ is decomposed into two genus one knots.
Since the degree of the Conway polynomial $\nabla_{K_{n}}(z)$ of $K_{n}$ is four, $\nabla_{K_{n}}(z)$ is decomposed into two Conway polynomials with degree two, that is, $\nabla_{K_{n}}(z)=(pz^{2}+1)(qz^{2}+1)$ for some integers $p$ and $q$.
Therefore we have
\[
pq=-(4n+1)\quad
\mbox{and}\quad
p+q=2n.
\]
Then $p$ and $q$ are integral roots of the equation $x^{2}-2nx-(4n+1)=0$, and hence the discriminant of the equation is a square integer.
Hence there is a non-negative integer $s$ such that $n^{2}+4n+1=s^{2}$.
Since it implies that $(n+2+s)(n+2-s)=3$, we have $s=1$, and $n=0$ and $-4$.

\end{proof}
\begin{figure}[h]
\begin{minipage}{75mm}
 \centering
 \includegraphics{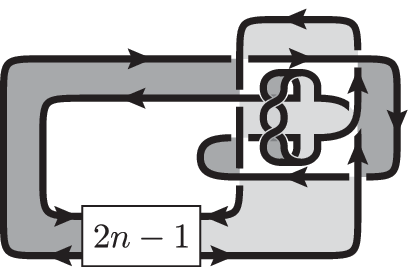}
 \caption{A Seifert surface of $K_{n}$}
 \label{fig:Seifert_surface}
\end{minipage}
\begin{minipage}{75mm}
 \centering
 \includegraphics{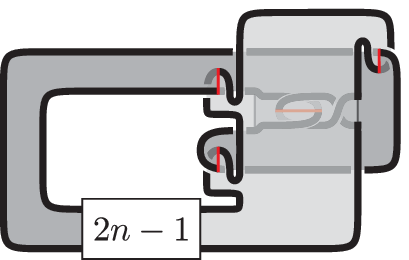}
 \caption{A clasp disk of $K_{n}$}
 \label{fig:c=4}
\end{minipage}
\end{figure}
\begin{rem}
(1)~On $u(K_{n})$, by using the Nakanishi index (cf.\ \cite{Kw1}) reduced modulo $2$, we can see that $u(K_{n})=2$ for even $n$.
The authors conjecture that $u(K_{n})=2$ for all $n$.

(2)~On $c(K_{n})$, we see that $c(K_{0})=2$ and $c(K_{1})=3$.
The authors conjecture that $c(K_{n})=4$ for any integer $n$ except $0~\mbox{and}~1$.

(3)~On the primeness of $K_{n}$, the authors conjecture that $K_{-4}$ is also prime.
\end{rem}

Now, we prove Theorem~\ref{thm1}.

\begin{proof}[Proof of Theorem~\ref{thm1}]
Suppose that $n$ is odd.
By Lemma~\ref{lem1}, the knot $K_{n}$ is prime and we have $\max\{g(K_{n}),u(K_{n})\}<c(K_{n})$.

\end{proof}

We show that Question~\ref{que2} is affirmative for the knots $K_{n}$.
\begin{lem}\label{lem2}
For the knots $K_{n}$ $(n\in\mathbb{Z})$, $cr(K_{n})=2n+8$ for positive $n$, $cr(K_{0})=7$, and $-2n+7\le cr(K_{n})\le -2n+10$ for negative $n$.
\end{lem}
\begin{proof}
{\it The case of $n>0$}:
Since we can see that $K_{n}$ has a reduced alternating diagram (see Figure~\ref{fig:ori-k_n}), by \cite[Theorem 13.5]{Mu} we have $cr(K_{n})=2n+8$.

\smallskip
{\it The case of $n=0$}:
Since $K_{0}=\overline{3_{1}}\#4_{1}$, by \cite[Theorem 13.5]{Mu} we have $cr(K_{0})=7$.

\smallskip
{\it The case of $n<0$}:
It is known that the degree of the Jones polynomial of a knot is a lower bound of the crossing number of the knot.
The Jones polynomial of $K_{n}$ is as follows:
\[
V_{K_{n}}(t)=t^{2n}V_{K_{0}}(t)-t^{-1}(t^{\frac{1}{2}}-t^{-\frac{1}{2}})V_{J}(t)(1+t^{-2}+\dots+t^{2(n+1)}),
\]
where
\[
V_{K_{0}}(t)=t^{-1}-1+2t-3t^{2}+3t^{3}-2t^{4}+2t^{5}-t^{6}
\]
and
\[
V_{J}(t)=-t^{-\frac{3}{2}}+2t^{-\frac{1}{2}}-4t^{\frac{1}{2}}+6t^{\frac{3}{2}}-6t^{\frac{5}{2}}+5t^{\frac{7}{2}}-6t^{\frac{9}{2}}+3t^{\frac{11}{2}}-2t^{\frac{13}{2}}+t^{\frac{15}{2}}.
\]
Therefore we have $cr(K_{n})\geq7-2n$.
On the other hand, it is easy to see that $cr(K_{n})\leq-2n+10$.
\end{proof}
By Lemma~\ref{lem1} (3) and Lemma~\ref{lem2}, we have the following.
\begin{prop}\label{prop:que2}
For the knots $K_{n}$ $(n\in\mathbb{Z})$, the following holds.
\[
c(K_{n})\leq \left[\frac{cr(K_{n})-1}{2}\right].
\]
\end{prop}
This is a supporting evidence that the question is affirmative.

%%%%%%%%%%%%%%%%%%%%%%%%%%%%%%%%%%%%%%%%%%%%%%%%%%%%%%%%%%%%%%%%%%%%%%%%%%%%%%%%%%%%%%%%%%%%%%%%%%

\section*{Appendix}

In the table below, $g$ is the genus, $u$ is the unknotting number, $c$ is the clasp number, and $X=2$ or 3 (cf.\ \cite{Kw1, Kw2, Ro}).
The clasp numbers of torus knots (\cite{GHY,Mo2}) and doubled knots (cf.\ \cite{Ko2,ST}) were determined. 
We refer the newest information of the unknotting numbers mainly from KnotInfo \cite{CL1}.
In the table, we are using the notation of prime knots in Rolfsen's table \cite{Ro}.
(We note that 
(1)~$10_{161}$ is equivalent to $10_{162}$ in Rolfsen's table and Kawauchi's table \cite{Kw1, Kw2}, 
(2)~in Kawauchi's table, $10_{83}$ and $10_{86}$ are interchanged from Rolfsen's table, and
(3)~Knot Atlas \cite{BM} and KnotInfo follow basically Rolfsen's table, but they remove $10_{162}$ in Rolfsen's table and denote $10_{n}$~$(163\le n\le 166)$ in Rolfsen's table by $10_{n-1}$.)

%%%%% 謝辞 %%%%%

\section*{Acknowledgements}
The authors would like to thank S.~Kamada, H.~Matsuda, K.~Morimoto, K.~Taniyama and I.~Tayama for their helpful comments.

%%%%% A table of the clasp numbers of prime knots with up to $10$ crossings %%%%%
\newpage

\begin{minipage}[t]{4.5cm}
 \begin{center}
  \begin{tabular}{|c@{\quad\vrule width0.8pt\quad}c|c|c|c|}
    \hline
  knot & $g$ & $u$ & $c$ \\
  \noalign{\hrule height 0.8pt}
  $3_1$ & 1 & 1  & 1 \\
    \hline
  $4_1$ & 1 & 1  & 1 \\
    \hline
  $5_1$ & 2 & 2  & 2 \\
    \hline
  $5_2$ & 1 & 1  & 1 \\
    \hline
  $6_1$ & 1 & 1  & 1 \\
    \hline
  $6_2$ & 2 & 1  & 2 \\
    \hline
  $6_3$ & 2 & 1  & 2 \\
    \hline
  $7_1$ & 3 & 3  & 3 \\
    \hline
  $7_2$ & 1 & 1  & 1 \\
    \hline
  $7_3$ & 2 & 2  & 2 \\
    \hline
  $7_4$ & 1 & 2  & 2 \\
    \hline
  $7_5$ & 2 & 2  & 2 \\
    \hline
  $7_6$ & 2 & 1  & 2 \\
    \hline
  $7_7$ & 2 & 1  & 2 \\
    \hline
  $8_1$ & 1 & 1  & 1 \\
    \hline
  $8_2$ & 3 & 2 & 3 \\
    \hline
  $8_3$ & 1 & 2 & 2 \\
    \hline
  $8_4$ & 2 & 2 & 2 \\
    \hline
  $8_5$ & 3 & 2 & 3 \\
    \hline
  $8_6$ & 2 & 2 & 2 \\
    \hline
  $8_7$ & 3 & 1 & 3 \\
    \hline
  $8_8$ & 2 & 2 & 2 \\
    \hline
  $8_9$ & 3 & 1 & 3 \\
    \hline
  $8_{10}$ & 3 & 2 & 3 \\
    \hline
  $8_{11}$ & 2 & 1 & 2 \\
    \hline
  $8_{12}$ & 2 & 2 & 2 \\
    \hline
  $8_{13}$ & 2 & 1 & 2 \\
    \hline
  $8_{14}$ & 2 & 1 & 2 \\
    \hline
  $8_{15}$ & 2 & 2 & 2 \\
    \hline
  $8_{16}$ & 3 & 2 & 3 \\
    \hline
  $8_{17}$ & 3 & 1 & 3 \\
    \hline
  $8_{18}$ & 3 & 2 & 3 \\
    \hline
  $8_{19}$ & 3 & 3 & 3 \\
    \hline
  $8_{20}$ & 2 & 1 & 2 \\
    \hline
  $8_{21}$ & 2 & 1 & 2 \\
    \hline
  $9_1$ & 4 & 4 & 4 \\
    \hline
  $9_2$ & 1 & 1 & 1 \\
    \hline
  $9_3$ & 3 & 3 & 3 \\
    \hline
  $9_4$ & 2 & 2 & 2 \\
    \hline
  $9_5$ & 1 & 2 & 2 \\
    \hline
  $9_6$ & 3 & 3 & 3 \\
    \hline
  $9_7$ & 2 & 2 & 2 \\
    \hline
  \end{tabular}
 \end{center}
\end{minipage}
\hspace{0.5cm}
\begin{minipage}[t]{4.5cm}
 \begin{center}
  \begin{tabular}{|c@{\quad\vrule width0.8pt\quad}c|c|c|c|}
    \hline
  knot & $g$ & $u$ & $c$ \\
  \noalign{\hrule height 0.8pt}
  $9_8$ & 2 & 2 & 2 \\
    \hline
  $9_9$ & 3 & 3 & 3 \\
    \hline
  $9_{10}$ & 2 & 3 & 3 \\
    \hline
  $9_{11}$ & 3 & 2 & 3 \\
    \hline
  $9_{12}$ & 2 & 1 & 2 \\
    \hline
  $9_{13}$ & 2 & 3 & 3 \\
    \hline
  $9_{14}$ & 2 & 1 & 2 \\
    \hline
  $9_{15}$ & 2 & 2 & 2 \\
    \hline
  $9_{16}$ & 3 & 3 & 3 \\
    \hline
  $9_{17}$ & 3 & 2 & 3 \\
    \hline
  $9_{18}$ & 2 & 2 & 2 \\
    \hline
  $9_{19}$ & 2 & 1 & 2 \\
    \hline
  $9_{20}$ & 3 & 2 & 3 \\
    \hline
  $9_{21}$ & 2 & 1 & 2 \\
    \hline
  $9_{22}$ & 3 & 1 & 3 \\
    \hline
  $9_{23}$ & 2 & 2 & 2 \\
    \hline
  $9_{24}$ & 3 & 1 & 3 \\
    \hline
  $9_{25}$ & 2 & 2 & 2 \\
    \hline
  $9_{26}$ & 3 & 1 & 3 \\
    \hline
  $9_{27}$ & 3 & 1 & 3 \\
    \hline
  $9_{28}$ & 3 & 1 & 3 \\
    \hline
  $9_{29}$ & 3 & 2 & 3 \\
    \hline
  $9_{30}$ & 3 & 1 & 3 \\
    \hline
  $9_{31}$ & 3 & 2 & 3 \\
    \hline
  $9_{32}$ & 3 & 2 & 3 \\
    \hline
  $9_{33}$ & 3 & 1 & 3 \\
    \hline
  $9_{34}$ & 3 & 1 & 3 \\
    \hline
  $9_{35}$  & 1 & 3 & 3 \\
    \hline
  $9_{36}$ & 3 & 2 & 3 \\
    \hline
  $9_{37}$ & 2 & 2 & 2 \\
    \hline
  $9_{38}$ & 2 & 3 & 3 \\
    \hline
  $9_{39}$ & 2 & 1 & $X$ \\
    \hline
  $9_{40}$ & 3 & 2 & 3 \\
    \hline
  $9_{41}$ & 2 & 2 & $X$ \\
    \hline
  $9_{42}$ & 2 & 1 & 2 \\
    \hline
  $9_{43}$ & 3 & 2 & 3 \\
    \hline
  $9_{44}$ & 2 & 1 & 2 \\
    \hline
  $9_{45}$ & 2 & 1 & 2 \\
    \hline
  $9_{46}$ & 1 & 2 & 2 \\
    \hline
  $9_{47}$ & 3 & 2 & 3 \\
    \hline
  $9_{48}$ & 2 & 2 & 2 \\
    \hline
  $9_{49}$ & 2 & 3 & 3 \\
    \hline
  \end{tabular}
 \end{center}
\end{minipage}
\hspace{0.5cm}
\begin{minipage}[t]{4.5cm}
 \begin{center}
  \begin{tabular}{|c@{\quad\vrule width0.8pt\quad}c|c|c|c|}
    \hline
  knot & $g$ & $u$ & $c$ \\
  \noalign{\hrule height 0.8pt}
  $10_1$ & 1 & 1 & 1 \\
    \hline
  $10_2$ & 4 & 3 & 4 \\
    \hline
  $10_3$ & 1 & 2 & 2 \\
    \hline
  $10_4$ & 2 & 2 & 2 \\
    \hline
  $10_5$ & 4 & 2 & 4 \\
    \hline
  $10_6$ & 3 & 3 & 3 \\
    \hline
  $10_7$ & 2 & 1 & 2 \\
    \hline
  $10_8$ & 3 & 2 & 3 \\
    \hline
  $10_9$ & 4 & 1 & 4 \\
    \hline
  $10_{10}$ & 2 & 1 & 2 \\
    \hline
  $10_{11}$ & 2 & $X$ & $X$ \\
    \hline
  $10_{12}$ & 3 & 2 & 3 \\
    \hline
  $10_{13}$ & 2 & 2 & 2 \\
    \hline
  $10_{14}$ & 3 & 2 & 3 \\
    \hline
  $10_{15}$ & 3 & 2 & 3 \\
    \hline
  $10_{16}$ & 2 & 2 & $X$ \\
    \hline
  $10_{17}$ & 4 & 1 & 4 \\
    \hline
  $10_{18}$ & 2 & 1 & 2 \\
    \hline
  $10_{19}$ & 3 & 2 & 3 \\
    \hline
  $10_{20}$ & 2 & 2 & 2 \\
    \hline
  $10_{21}$ & 3 & 2 & 3 \\
    \hline
  $10_{22}$ & 3 & 2 & 3 \\
    \hline
  $10_{23}$ & 3 & 1 & 3 \\
    \hline
  $10_{24}$ & 2 & 2 & 2 \\
    \hline
  $10_{25}$ & 3 & 2 & 3 \\
    \hline
  $10_{26}$ & 3 & 1 & 3 \\
    \hline
  $10_{27}$ & 3 & 1 & 3 \\
    \hline
  $10_{28}$ & 2 & 2 & $X$ \\
    \hline
  $10_{29}$ & 3 & 2 & 3 \\
    \hline
  $10_{30}$ & 2 & 1 & $X$ \\
    \hline
  $10_{31}$ & 2 & 1 & 2 \\
    \hline
  $10_{32}$ & 3 & 1 & 3 \\
    \hline
  $10_{33}$ & 2 & 1 & $X$ \\
    \hline
  $10_{34}$ & 2 & 2 & 2 \\
    \hline
  $10_{35}$ & 2 & 2 & 2 \\
    \hline
  $10_{36}$ & 2 & 2 & 2 \\
    \hline
  $10_{37}$ & 2 & 2 & 2 \\
    \hline
  $10_{38}$ & 2 & 2 & 2 \\
    \hline
  $10_{39}$ & 3 & 2 & 3 \\
    \hline
  $10_{40}$ & 3 & 2 & 3 \\
    \hline
  $10_{41}$ & 3 & 2 & 3 \\
    \hline
  $10_{42}$ & 3 & 1 & 3 \\
    \hline
  \end{tabular}
 \end{center}
\end{minipage}

\newpage

\begin{minipage}[t]{4.5cm}
 \begin{center}
  \begin{tabular}{|c@{\quad\vrule width0.8pt\quad}c|c|c|c|}
    \hline
  knot & $g$ & $u$  & $c$ \\
  \noalign{\hrule height 0.8pt}
  $10_{43}$ & 3 & 2 & 3 \\
    \hline
  $10_{44}$ & 3 & 1 & 3 \\
    \hline
  $10_{45}$ & 3 & 2 & 3 \\
    \hline
  $10_{46}$ & 4 & 3 & 4 \\
    \hline
  $10_{47}$ & 4 & $X$ & 4 \\
    \hline
  $10_{48}$ & 4 & 2 & 4 \\
    \hline
  $10_{49}$ & 3 & 3 & 3 \\
    \hline
  $10_{50}$ & 3 & 2 & 3 \\
    \hline
  $10_{51}$ & 3 & $X$ & 3 \\
    \hline
  $10_{52}$ & 3 & 2 & 3 \\
    \hline
  $10_{53}$ & 2 & 3 & 3 \\
    \hline
  $10_{54}$ & 3 & $X$ & 3 \\
    \hline
  $10_{55}$ & 2 & 2 & 2 \\
    \hline
  $10_{56}$ & 3 & 2 & 3 \\
    \hline
  $10_{57}$ & 3 & 2 & 3 \\
    \hline
  $10_{58}$ & 2 & 2 & 2 \\
    \hline
  $10_{59}$ & 3 & 1 & 3 \\
    \hline
  $10_{60}$ & 3 & 1 & 3 \\
    \hline
  $10_{61}$ & 3 & $X$ & 3 \\
    \hline
  $10_{62}$ & 4 & 2 & 4 \\
    \hline
  $10_{63}$ & 2 & 2 & 2 \\
    \hline
  $10_{64}$ & 4 & 2 & 4 \\
    \hline
  $10_{65}$ & 3 & 2 & 3 \\
    \hline
  $10_{66}$ & 3 & 3 & 3 \\
    \hline
  $10_{67}$ & 2 & 2 & 2 \\
    \hline
  $10_{68}$ & 2 & 2 & $X$ \\
    \hline
  $10_{69}$ & 3 & 2 & 3 \\
    \hline
  $10_{70}$ & 3 & 2 & 3 \\
    \hline
  $10_{71}$ & 3 & 1 & 3 \\
    \hline
  $10_{72}$ & 3 & 2 & 3 \\
    \hline
  $10_{73}$ & 3 & 1 & 3 \\
    \hline
  $10_{74}$ & 2 & 2 & $X$ \\
    \hline
  $10_{75}$ & 3 & 2 & 3 \\
    \hline
  $10_{76}$ & 3 & $X$ & 3 \\
    \hline
  $10_{77}$ & 3 & $X$ & 3 \\
    \hline
  $10_{78}$ & 3 & 2 & 3 \\
    \hline
  $10_{79}$ & 4 & $X$ & 4 \\
    \hline
  $10_{80}$ & 3 & 3 & 3 \\
    \hline
  $10_{81}$ & 3 & 2 & 3 \\
    \hline
  $10_{82}$ & 4 & 1 & 4 \\
    \hline
  $10_{83}$ & 3 & 2 & 3 \\
    \hline
  $10_{84}$ & 3 & 1 & 3 \\
    \hline
  \end{tabular}
 \end{center}
\end{minipage}
\hspace{0.5cm}
\begin{minipage}[t]{4.5cm}
 \begin{center}
  \begin{tabular}{|c@{\quad\vrule width0.8pt\quad}c|c|c|c|}
    \hline
  knot & $g$ & $u$ & $c$ \\
  \noalign{\hrule height 0.8pt}
  $10_{85}$ & 4 & 2 & 4 \\
    \hline
  $10_{86}$ & 3 & 2 & 3 \\
    \hline
  $10_{87}$ & 3 & 2 & 3 \\
    \hline
  $10_{88}$ & 3 & 1 & 3 \\
    \hline
  $10_{89}$ & 3 & 2 & 3 \\
    \hline
  $10_{90}$ & 3 & 2 & 3 \\
    \hline
  $10_{91}$ & 4 & 1 & 4 \\
    \hline
  $10_{92}$ & 3 & 2 & 3 \\
    \hline
  $10_{93}$ & 3 & 2 & 3 \\
    \hline
  $10_{94}$ & 4 & 2 & 4 \\
    \hline
  $10_{95}$ & 3 & 1 & 3 \\
    \hline
  $10_{96}$ & 3 & 2 & 3 \\
    \hline
  $10_{97}$ & 2 & 2 & 3 \\
    \hline
  $10_{98}$ & 3 & 2 & 3 \\
    \hline
  $10_{99}$ & 4 & 2 & 4 \\
    \hline
  $10_{100}$ & 4 & $X$ & 4 \\
    \hline
  $10_{101}$ & 2 & 3 & 3 \\
    \hline
  $10_{102}$ & 3 & 1 & 3 \\
    \hline
  $10_{103}$ & 3 & 3 & 3 \\
    \hline
  $10_{104}$ & 4 & 1 & 4 \\
    \hline
  $10_{105}$ & 3 & 2 & 3 \\
    \hline
  $10_{106}$ & 4 & 2 & 4 \\
    \hline
  $10_{107}$ & 3 & 1 & 3 \\
    \hline
  $10_{108}$ & 3 & 2 & 3 \\
    \hline
  $10_{109}$ & 4 & 2 & 4 \\
    \hline
  $10_{110}$ & 3 & 2 & 3 \\
    \hline
  $10_{111}$ & 3 & 2 & 3 \\
    \hline
  $10_{112}$ & 4 & 2 & 4 \\
    \hline
  $10_{113}$ & 3 & 1 & 3 \\
    \hline
  $10_{114}$ & 3 & 1 & 3 \\
    \hline
  $10_{115}$ & 3 & 2 & 3 \\
    \hline
  $10_{116}$ & 4 & 2 & 4 \\
    \hline
  $10_{117}$ & 3 & 2 & 3 \\
    \hline
  $10_{118}$ & 4 & 1 & 4 \\
    \hline
  $10_{119}$ & 3 & 1 & 3 \\
    \hline
  $10_{120}$ & 2 & 3 & 3 \\
    \hline
  $10_{121}$ & 3 & 2 & 3 \\
    \hline
  $10_{122}$ & 3 & 2 & 3 \\
    \hline
  $10_{123}$ & 4 & 2 & 4 \\
    \hline
  $10_{124}$ & 4 & 4 & 4 \\
    \hline
  $10_{125}$ & 3 & 2 & 3 \\
    \hline
  $10_{126}$ & 3 & 2 & 3 \\
    \hline
  \end{tabular}
 \end{center}
\end{minipage}
\hspace{0.5cm}
\begin{minipage}[t]{4.5cm}
 \begin{center}
  \begin{tabular}{|c@{\quad\vrule width0.8pt\quad}c|c|c|c|}
    \hline
  knot & $g$ & $u$ & $c$ \\
  \noalign{\hrule height 0.8pt}
  $10_{127}$ & 3 & 2 & 3 \\
    \hline
  $10_{128}$ & 3 & 3 & 3 \\
    \hline
  $10_{129}$ & 2 & 1 & $X$ \\
    \hline
  $10_{130}$ & 2 & 2 & $X$ \\
    \hline
  $10_{131}$ & 2 & 1 & $X$ \\
    \hline
  $10_{132}$ & 2 & 1 & 2\\
    \hline
  $10_{133}$ & 2 & 1 & 2 \\
    \hline
  $10_{134}$ & 3 & 3 & 3 \\
    \hline
  $10_{135}$ & 2 & 2 & 2 \\
    \hline
  $10_{136}$ & 2 & 1 & 2 \\
    \hline
  $10_{137}$ & 2 & 1 & 2\\
    \hline
  $10_{138}$ & 3 & 2 & 3 \\
    \hline
  $10_{139}$ & 4 & 4 & 4 \\
    \hline
  $10_{140}$ & 2 & 2 & 2 \\
    \hline
  $10_{141}$ & 3 & 1 & 3 \\
    \hline
  $10_{142}$ & 3 & 3 & 3 \\
    \hline
  $10_{143}$ & 3 & 1 & 3 \\
    \hline
  $10_{144}$ & 2 & 2 & 2 \\
    \hline
  $10_{145}$ & 2 & 2 & 2 \\
    \hline
  $10_{146}$ & 2 & 1 & 2 \\
    \hline
  $10_{147}$ & 2 & 1 & 2 \\
    \hline
  $10_{148}$ & 3 & 2 & 3 \\
    \hline
  $10_{149}$ & 3 & 2 & 3 \\
    \hline
  $10_{150}$ & 3 & 2 & 3 \\
    \hline
  $10_{151}$ & 3 & 2 & 3 \\
    \hline
  $10_{152}$ & 4 & 4 & 4 \\
    \hline
  $10_{153}$ & 3 & 2 & 3 \\
    \hline
  $10_{154}$ & 3 & 3 & 3 \\
    \hline
  $10_{155}$ & 3 & 2 & 3 \\
    \hline
  $10_{156}$ & 3 & 1 & 3 \\
    \hline
  $10_{157}$ & 3 & 2 & 3 \\
    \hline
  $10_{158}$ & 3 & 2 & 3 \\
    \hline
  $10_{159}$ & 3 & 1 & 3 \\
    \hline
  $10_{160}$ & 3 & 2 & 3 \\
    \hline
  $10_{161}$ & 3 & 3 & 3 \\
    \hline
  $10_{162}$ & 3 & 3 & 3 \\
    \hline
  $10_{163}$ & 2 & 2 & $X$ \\
    \hline
  $10_{164}$ & 3 & 2 & 3 \\
    \hline
  $10_{165}$ & 2 & 1 & $X$ \\
    \hline
  $10_{166}$ & 2 & 2 & $X$ \\
    \hline  
             &   &   &     \\
    %\hline  
             &   &   &     \\
    \hline  
  \end{tabular}
 \end{center}
\end{minipage}

%%%%% 参考文献 %%%%%

\end{document}